\documentclass[12pt,reqno]{amsart}
\usepackage{amsmath, amssymb, latexsym, amscd, amsthm,amsfonts,amstext}
\usepackage[mathscr]{eucal}
\theoremstyle{plain}
\newtheorem{theorem}{Theorem}

\newtheorem{Lemma}[theorem]{Lemma}
\newtheorem{proposition}[theorem]{Proposition}
\theoremstyle{remark}

\newtheorem{nx}{Remark}
\theoremstyle{definition}
\newtheorem{define}{Definition}

\def\R{{\mathscr R}}

\def\CC{{\mathbb C}}

\def\H{{\mathscr H}}
\def\P{{\mathscr P}}

\def\leq{\leqslant}

\def\RR{{\mathbb R}}

\begin{document}
\date{}

\title[Characterization of domains in $\CC^n$ ]{Characterization of domains in $\CC^n$ by their noncompact automorphism groups}

\author{Do Duc Thai and Ninh Van Thu} 
\maketitle      
\begin{abstract}
In this paper, the characterization of domains in $\CC^n$ by their noncompact automorphism groups
are given. 
\end{abstract}
\tableofcontents

\section{Introduction}  
Let $\Omega$ be a domain, i.e. connected open subset, in a complex manifold $M$. Let the {\it automorphism group} of $\Omega$ (denoted
$Aut(\Omega)$) be the collection of biholomorphic self-maps of $\Omega$ with composition of mappings as its binary operation. The
topology on $Aut(\Omega)$ is that of uniform convergence on compact sets (i.e., the compact-open topology).

One of the important problems in several complex variables is to study
the interplay between the geometry of a domain and  the structure of its
automorphism group. More precisely, we wish to see to what extent a 
domain is determined by its automorphism group.

It is a standard and classical result of H. Cartan that if $\Omega$ is a bounded domain in $\CC^n$ and the automorphism group of $\Omega$ is noncompact then there exist a point $x \in \Omega$, a point $p \in \partial \Omega$, and automorphisms $\varphi_j \in Aut(\Omega)$ such that $\varphi_j(x) \to p$. In this circumstance we call $p$ a {\it boundary orbit accumulation point}. 

Works in the past twenty years has suggested that the local geometry of the so-called ''boundary orbit accumulation point'' $p$ in turn gives global information about the characterization of model of the domain. We refer readers to the recent survey \cite{IK} and the references therein for the development in related subjects. For instance, B. Wong and J. P. Rosay (see \cite{W}, \cite{R}) proved the following theorem.

{\bf Wong-Rosay theorem.} {\it Any bounded domain $\Omega \Subset\CC^n$ with a $C^2$ strongly pseudoconvex boundary orbit accumulation point is biholomorphic to the unit ball in $\CC^n$.  }

By using the scaling technique, introduced by S. Pinchuk \cite {P},   E. Bedford and S. Pinchuk \cite{B-P2} proved the theorem about the characterization of the complex ellipsoids.

{\bf Bedford-Pinchuk theorem.} {\it Let $\Omega\subset \CC^{n+1}$ be a bounded pseudoconvex domain of finite type whose boundary is smooth of class $C^\infty$, and suppose that the Levi  form has rank at least $n-1$ at each point of the boundary. If $Aut(\Omega)$ is noncompact, then $\Omega$ is biholomorphically equivalent to the domain 
$$E_{m}=\{(w, z_1,\cdots, z_n)\in \CC^{n+1}: |w|^2+ |z_1|^{2m}+|z_2|^2+\cdots+|z_n|^2<1\},$$
for some integer $m\geq 1$.}

 We would like to emphasize here that the assumption on boundedness of domains in the above-mentioned theorem is essential in their proofs. It seems to us that some key techniques in their proofs could not use for unbounded domains in $\CC^n$. Thus, there is a natural question that whether the Bedford-Pinchuk theorem is true for any domain in $\CC^n$.
In 1994, F. Berteloot   \cite{Ber2} gave a partial answer to this question in dimension 2. 

{\bf Berteloot theorem.} {\it Let $\Omega$ be a domain in $\CC^2$ and let $\xi_0\in \partial \Omega$. 
Assume that there exists a sequence $(\varphi_p)$ in $Aut(\Omega)$ and a point $a\in \Omega$ such that $\lim\varphi_p(a)=\xi_0$. If  $\partial \Omega$ is pseudoconvex and of finite type near $\xi_0$ then $\Omega$ is biholomorphically equivalent to $\{(w,z)\in \CC^2: \R ew+ H(z,\bar z )<0\}$, where $H$ is a homogeneous subharmonic polynomial on $\CC$ with degree  $2m$.}

The main aim in this paper is to show that the above theorems of Bedford-Pinchuk and Berteloot hold for  domains (not necessary bounded) in  $\CC^n$. Namely, we prove the following.
\begin{theorem} \label{T} Let $\Omega$ be a domain in $\CC^n$ and let $\xi_0\in \partial \Omega $. Assume that
\begin{enumerate}
\item[(a)] $\partial \Omega$ is pseudoconvex, of finite type and smooth of class $C^\infty$ in some neighbourhood of $\xi_0 \in \partial \Omega$. 
\item[(b)] The Levi  form has rank at least $n-2$ at $\xi_0$. 
\item[(c)] There exists a sequence $(\varphi_p)$ in $Aut(\Omega)$ such that $\lim \varphi_p(a)=\xi_0$ for some $a\in \Omega $.
\end{enumerate}  
Then $\Omega$ is biholomorphically equivalent to a domain of the form
$$M_H=\{(w_1,\cdots,w_n)\in \CC^n:Re  w_n+H(w_1,\bar w_1)+\sum_{\alpha=2}^{n-1}|w_\alpha|^2<0\},$$
where $H$ is a homogeneous subharmonic polynomial with $\Delta H\not\equiv 0$.
\end{theorem}

Notations
\begin{itemize}
\item $\H(\omega, \Omega)$ is the set of holomorphic mappings from $\omega$ to $\Omega.$
\item $f_p$ is u.c.c on $\omega$  means that the sequence $(f_p)$, $f_p\in \H(\omega,\Omega)$, uniformly converges on compact subsets of $\omega.$
\item $\P_{2m}$ is the space of real valued polynomials on $\CC$ with degree less than $2m$ and which do not contain any harmonic terms.
\item $\H_{2m}=\{H\in \P_{2m}\  \text{ such that }\ \deg H=2m  \text{ and }\ H \ \text{ is homogeneous }\\ \text{ and subharmonic }\}.$
\item  $M_Q=\{z\in \CC^n: Re z_n+Q(z_1)+|z_2|^2+\cdots+|z_{n-1}|^2<0\}$ where $Q\in \P_{2m}.$
\item  $\Omega_1 \simeq\Omega_2$ means that $\Omega_1$ and $\Omega_2$ are biholomorphic equivalent.
\end{itemize}
The paper is organized as follows. In Section 2, we review some basic notions needed later. In Section 3, we discribe the construction of polydiscs around points near the boundary of a domain, and give some of their properties. In particular, we use the Scaling method to show that $\Omega$ is biholomorphic to a model $M_P$ with $P\in \P_{2m}$. In Section 4, we end the proof of our theorem by using the Berteloot's method.

{\bf Acknowledgement.} We would like to thank Professor Fran\c{c}ois Berteloot for his precious discusions on this material. Especially, we would like to express our gratitude to the refree. His/her valuable comments on the first version of this paper led to significant improvements.

\section{Definitions and results}
\setcounter{theorem}{0}
\setcounter{nx}{0}
\setcounter{define}{0}
\setcounter{equation}{0}

First of all, we recall the following definition (see \cite{GK}). 
\begin{define} Let $\{\Omega_i\}_{i=1}^\infty$ be a sequence of open sets in a complex manifold $M$ and $\Omega_0 $ be an open set of $M$. The sequence $\{\Omega_i\}_{i=1}^\infty$ is said to converge to $\Omega_0 $, written $\lim\Omega_i=\Omega_0$ iff 
\begin{enumerate}
\item[(i)] For any compact set $K\subset \Omega_0,$ there is a $i_0=i_0(K)$ such that $i\geq i_0$ implies $K\subset \Omega_i$, and 
\item[(ii)] If $K$ is a compact set which is contained in $\Omega_i$ for all sufficiently large $i,$ then  $K\subset \Omega_0$.
\end{enumerate}  
\end{define}

The following proposition is the generalization of the theorem of H. Cartan (see \cite{GK}, \cite{TM} for more generalizations of this theorem).
\begin{proposition} \label{T:7}  Let $\{A_i\}_{i=1}^\infty$ and $\{\Omega_i\}_{i=1}^\infty$ be sequences of domains in a complex manifold $M$ with $\lim A_i=A_0$ and $\lim \Omega_i=\Omega_0$ for some (uniquely determined) domains $A_0$, $\Omega_0$ in $M$. Suppose that $\{f_i: A_i \to \Omega_i\} $ is a sequence of biholomorphic maps. Suppose also that the sequence $\{f_i: A_i\to M \}$ converges uniformly on compact subsets of $ A_0$ to a holomorphic map $F:A_0\to M $ and the sequence $\{g_i:=f^{-1}_i: \Omega_i\to M \}$ converges uniformly on compact subsets of $\Omega_0$ to a holomorphic map $G:\Omega_0\to M $.Then one of the following two assertions holds. 
\begin{enumerate}
\item[(i)] The sequence $\{f_i\}$ is compactly divergent, i.e., for each compact set $K\subset \Omega_0$ and each compact set $L\subset \Omega_0$, there exists an integer $i_0$ such that $f_i(K)\cap L=\emptyset$ for $i\geq i_0$, or
\item[(ii)] There exists a subsequence $\{f_{i_j}\}\subset \{f_i\}$  such that the sequence $\{f_{i_j}\}$ converges uniformly on compact subsets of $A_0$ to a biholomorphic map $F: A_0 \to \Omega_0$.
\end{enumerate}
\end{proposition}
\begin{proof} 
Assume that the sequence $\{f_i\}$ is not divergent. Then $F$ maps some point $p$ of $A_0$ into $\Omega_0$. We will show that $F$ is a biholomorphism of   $A_0 $ onto $\Omega_0$. Let $q=F(p).$ Then
$$G(q)=G(F(p))=\lim_{i\to \infty} g_i(F(p))=\lim_{i\to \infty} g_i(f_i(p))=p.$$

Take a neighbourhood $V$ of $p$ in $A_0$ such that $F(V)\subset \Omega_0$. But then uniform convergence allows us to conclude that, for all $z\in V,$ it holds that $G(F(z))=\lim_{i\to \infty} g_i(f_i(z))=z$. Hence $F_{\mid V} $ is injective. By the Osgood's theorem, the mapping $F_{\mid V}: V\to F(V) $ is biholomorphic. 

Consider the holomorphic functions $J_i: A_i\to \CC$ and $J:A_0\to \CC$ given by $J_i(z)=det((df_i)_z) $ and $J(z)=det((dF)_z) $. Then $J(z)\ne 0\ (z\in V)$ and, for each $i=1,2,\cdots$, the function $J_i$ is non-vanishing on $A_i$. Moreover, the sequence $\{J_i\}_{i=0}^\infty$  converges uniformly on compact subsets of $A_0$ to $J$. By Hurwitz's theorem, it follows that $J$ never vanishes. This implies that the mapping $F: A_0\to M$ is open and any $z\in A_0$ is isolated in $F^{-1}(F(z))$. According to Proposition 5 in \cite{N}, we have $F(A_0)\subset \Omega_0$. 

Of course this entire argument may be repeated to see that $G(\Omega_0)\subset A_0$. But then uniform convergence allows us to conclude that, for all $z\in A_0,$ it holds that $G\circ F(z)=\lim_{i\to \infty}g_i(f_i(z))=z$ and likewise for all $w\in \Omega_0$ it holds that $F\circ G(w)=\lim_{i\to \infty}f_i(g_i(w))=w$.    

This proves that $F$ and $G$ are each one-to-one and onto, hence in particular that $F$ is a biholomorphic mapping.  
\end{proof}
 
Next, by Proposition 2.1 in \cite{Ber2}, we have the following.

\begin{proposition}\label{T:8}  Let $M$ be a domain  in a complex manifold $X$ of dimension $n$ and $\xi_0\in \partial M$. Assume that $\partial M$ is pseudoconvex and of finite type near $\xi_0$. 
\begin{enumerate}
\item[(a)] Let $\Omega $ be a domain  in a complex manifold $Y$ of dimension $m$. Then every sequence $\{\varphi_p\}\subset Hol(\Omega,M)$ converges unifomly on compact subsets of $\Omega $ to $\xi_0$ if and only if $\lim   \varphi_p(a)=\xi_0$ for some $a\in \Omega$.
\item[(b)] Assume, moreover, that there exists a sequence $\{\varphi_p\}\subset Aut(M)$  such that $\lim \varphi_p(a)=\xi_0 $ for some $a\in M$. Then $M$ is taut.
\end{enumerate}
\end{proposition}  

\begin{proof}
Since $\partial M$ is pseudoconvex and of finite type near $\xi_0\in \partial M$, there exists a local  peak plurisubharmonic function at $\xi_0 $ (see \cite{Cho}). Moreover,  since $\partial M $ is smooth and pseudoconvex near $\xi_0$, there exists a small ball $B$ centered at $\xi_0$ such that $B\cap M $ is hyperconvex and therefore is taut. The theorem  is deduced from  Proposition 2.1 in \cite{Ber2}.     
\end{proof}

\begin{nx} \label{R:3} By Proposition \ref{T:8} and by the hypothesis of Theorem \ref{T}, for each compact subset $K\subset M$ and each neighbourhood $U$ of $\xi_0$, there exists an integer $p_0$ such that $\varphi_p(K)\subset M\cap U $ for every $p\geq p_0$
\end{nx}

\begin{nx} \label{R:4} By Proposition \ref{T:8} and by the hypothesis of Theorem \ref{T},  $M$ is taut.
\end{nx}

The following lemma  is a slightly modification of Lemma 2.3 in \cite{Ber2}. 

\begin{Lemma}\label{L:4}
Let $\sigma_\infty$ be a subharmonic function of class $C^2$ on $\CC$ such that $\sigma_\infty(0)=0 $ and $\int_\CC \bar\partial \partial\sigma_\infty =+\infty$. Let $(\sigma_k)_k$ be a sequence of subharmonic functions on $\CC$ which  converges uniformly on compact subsets of $\CC$ to $\sigma_\infty$. Let $\omega$ be any domain in a complex manifold of dimension $m\ (m\geq 1)$ and let  $z_0$ be fixed in $\omega$. Denote by $M_k$ the domain in $ \CC^n$ defined by 
$$M_k=\{(z_1,z_2,\cdots, z_n)\in \CC^n: Imz_1+\sigma_k(z_2)+|z_3|^2+\cdots+|z_n|^2<0\}.$$ 

Then any sequence $h_k\in Hol(\omega, M_k)$ such that $\{h_k(z_0), k\geq 0\}\Subset M_\infty $ admits some subsequence which converges uniformly on compact subsets of $\omega$ to some element of $Hol(\omega, M_\infty)$. 
\end{Lemma}

\section{Estimates of Kobayashi metric of the domains in $\CC^n$ }
\setcounter{theorem}{0}
\setcounter{nx}{0}
\setcounter{define}{0}
\setcounter{equation}{0}
In this section we use the Catlin's argument in \cite{Cat} to study special coordinates and polydiscs. After that, we improve Berteloot's technique in \cite{Ber3} to construct a dilation sequence, estimate the Kobayashi metric and prove the normality of a family of holomorphic mappings.
\subsection{Special Coordinates and Polydiscs}
Let $\Omega$ be a domain in $\CC^n$. Suppose that $\partial \Omega $ is pseudoconvex, finite type and is smooth of class $C^\infty$ near a boundary point $\xi_0\in \partial \Omega$ and suppose that the Levi form has rank at least $n-2$ at $\xi_0$. We may assume that $\xi_0=0$ and the rank of Levi form at $\xi_0$ is exactly $n-2$. Let $r$ be a smooth definning function for  $\Omega $. Note that the type $m$ at $\xi_0$ is an even integer in this case. We also assume that $\dfrac{\partial r}{\partial z_n}(z)\ne 0$ for all $z$ in a small neighborhood $U$ about $\xi_0$. After a linear change of coordinates, we can find cooordinate functions $z_1,\cdots, z_n$ defined on $U$ such that 
\begin{equation}\label{Eq1} 
\begin{split}
L_n=\frac{\partial }{\partial z_n},\ L_j=\frac{\partial }{\partial z_j}+b_j\frac{\partial }{\partial z_n},\ L_j r\equiv 0,\ b_j(\xi_0)=0,\ j=1,\cdots, n-1,
\end{split}
\end{equation}
which form a basis of $\CC T^{(1,0)}(U)$ and satisfy
\begin{equation}\label{Eq2} 
\begin{split}
\partial \bar \partial r(q)(L_i,\bar L_j)=\delta_{ij},\  2\leq i,j\leq n-1,
\end{split}
\end{equation}
where $\delta_{ij}=1 $ if $i=j$ and $\delta_{ij}=0$ otherwise.

We want to show that about each point $z'=({z'}_1,\cdots,{z'}_n)$ in $U$, there is a polydisc of maximal size on which the function $r(z)$ changes by no more than some prescribed small number $\delta$. First, we construct the coodinates about $z'$ introduced by S. Cho (see also in \cite{Cho}). These coodinates will be used to define the polydisc.

Let us take the coordinate functions $z_1,\cdots, z_n$ about $\xi_0$ so that (\ref{Eq2}) holds. Therefore $|L_n r(z)|\geq c>0$ for all $z\in U$, and $\partial\bar\partial r(z)(L_i, \bar L_j)_{2\leq i,j\leq n-1}$ has $(n-2)$-positive eigenvalues in $U$ where 
\begin{equation*}
\begin{split}
L_n&=\frac{\partial }{\partial z_n},\quad \text{and}\\
 L_j&=\frac{\partial }{\partial z_j}-(\frac{\partial r}{\partial z_n})^{-1}\frac{\partial r(z') }{\partial z_j}\frac{\partial }{\partial z_n},\ j=1,\cdots, n-1.
\end{split}
\end{equation*}

For each $z'\in U$, define new coordinate functions $u_1,\cdots,u_n$ defined by $z=\varphi_1(u)$
\begin{equation*}
\begin{split}
z_n&={z'}_n+u_n- \sum_{j=1}^{n-1}\big[(\frac{\partial r}{\partial z_n})^{-1}\frac{\partial r(z') }{\partial z_j}\big ] u_j,\\
 z_j&={z'}_j+u_j,\ j=1,\cdots, n-1.
\end{split}
\end{equation*}

Then $L_j$ can be written as
$L_j=\frac{\partial }{\partial u_j}+{b'}_j\frac{\partial }{\partial u_n},\ j=1,\cdots, n-1,$ where ${b'}_j(z')=0$. In $u_1,\cdots, u_n$ coordinates, $A=\Big (\dfrac{\partial^2 r(z')}{\partial u_i\partial \bar u_j}\Big )_{2\leq i,j\leq n-1}$ is an Hermitian matrix and there is a unitary matrix $P=\big (P_{ij}\big )_{2\leq i,j\leq n-1}$ such that $P^* AP=D$, where $D$ is a diagonal matrix whose entries are positive eigenvalues of $A$.

Define $u=\varphi_2(v)$ by 
\begin{equation*}
\begin{split}
u_1&=v_1, \ u_n=v_n,\quad \text{and}\\
 u_j&=\sum_{k=2}^{n-1}\bar P_{jk} v_k,\ j=2,\cdots, n-1,
\end{split}
\end{equation*}
Then $\dfrac{\partial^2 r(z')}{\partial v_i\partial \bar v_j}=\lambda_i\delta_{ij},\ 2\leq i,j\leq n-1$, where $\lambda_i>0$ is an $i$-th entry of $D$ (we may assume that $\lambda_i\geq c >0$ in $U$ for all $i$). Next we define $v=\varphi_3(w)$ by 
\begin{equation*}
\begin{split}
v_1&=w_1, \ v_n=w_n,\quad \text{and}\\
 v_j&=\lambda_j w_j,\ j=2,\cdots, n-1,
\end{split}
\end{equation*}
Then $\dfrac{\partial^2 r(z')}{\partial w_i\partial \bar w_j}=\delta_{ij},\ 2\leq i,j\leq n-1$ and $r(w)$ can be written as 
\begin{equation}\label{Eq4} 
\begin{split}
r(w)&=r(z')+Re w_n+\sum_{\alpha=2}^{n-1}\sum_{1\leq j\leq \frac{m}{2}} Re \big [(a^\alpha_j w_1^j+b^\alpha_j \bar w_1^j)w_\alpha\big ]+Re \sum_{\alpha=2}^{n-1}c_\alpha w_\alpha^2\\
&+ \sum_{2\leq j+k\leq m} a_{j,k}w_1^j \bar  w_1^k+\sum_{\alpha=2}^{n-1}|w_\alpha|^2+ \sum_{\alpha=2}^{n-1} \sum_{\substack{j+k\leq \frac{m}{2}\\
 j,k>0}} Re( b^\alpha_{j,k} w_1^j \bar w_1^k  w_\alpha)\\
&+O(|w_n| |w|+|w^*|^2|w|+|w^*|^2|w_1|^{\frac{m}{2}+1}+|w_1|^{m+1}),
\end{split}
\end{equation}
where $w^*=(0,w_2,\cdots,w_{n-1},0)$. It is standard to perform the change of coordinates $w=\varphi_4(t)$
\begin{equation*}
\begin{split}
w_n&=t_n-\sum_{2\leq k\leq m}\frac{2}{k!} \frac{\partial^k r(0) }{\partial w_1^k} t_1^k -\sum_{\alpha=2}^{n-1}\sum_{1\leq k\leq m/2}\frac{2}{(k+1)!} \frac{\partial^{k+1} r(0) }{\partial w_\alpha \partial w_1^k} t_\alpha t_1^k-\sum_{\alpha=2}^{n-1} \frac{\partial^2 r(0) }{\partial w_\alpha^2}t_\alpha^2, \\
 w_j&=t_j,\ j=1,\cdots, n-1,
\end{split}
\end{equation*}
which serves to remove the pure terms from $(\ref{Eq4})$, i.e., it removes $w_1^k, \bar w_1^k, w^2_\alpha $ terms as well as $w_1^k w_\alpha , \bar w_1^k \bar w_\alpha$ terms from the summation in (\ref{Eq4}). 

We may also perform a change of coordinates $t=\varphi_5(\zeta)$ defined by
\begin{equation*}
\begin{split}
t_1&=\zeta_1,\ t_n=\zeta_n,\\
 t_\alpha &=\zeta_\alpha-\sum_{1\leq k\leq \frac{m}{2}}\frac{1}{(k+1)!} \frac{\partial^{k+1} r(0) }{\partial \bar t_\alpha\partial t_1^k} \zeta_1^k,\ \alpha=2,\cdots, n-1
\end{split}
\end{equation*}
to remove terms of the form $\bar w_1^j w_\alpha$ from the summation in (\ref{Eq4}) and hence $r(\zeta)$ has the desired expression as in (\ref{Eq3}) in $\zeta$-coordinates.

Thus, we obtain the following Proposition (see also  in \cite[Prop. 2.2, p. 806]{Cho1}). 
\begin{proposition}[S. Cho] \label{P31} 
For each $z' \in U$ and positive even integer $m$, there is a biholomorphism $\Phi_{z'}:\CC^n\to \CC^n, \ z=\Phi_{z'}^{-1}(\zeta_1,\cdots,\zeta_n) $ such that  
\begin{equation}\label{Eq3} 
\begin{split}
r(\Phi_{z'}^{-1}(\zeta))&=r(z')+Re \zeta_n+ \sum_{\substack{j+k\leq m\\
 j,k>0}} a_{jk}(z')\zeta_1^j \bar \zeta_1^k\\
&+\sum_{\alpha=2}^{n-1}|\zeta_\alpha|^2+ \sum_{\alpha=2}^{n-1}Re (( \sum_{\substack{j+k\leq \frac{m}{2}\\
 j,k>0}} b^\alpha_{jk}(z')\zeta_1^j \bar \zeta_1^k)\zeta_\alpha)\\
&+O(|\zeta_n| |\zeta|+|\zeta^*|^2|\zeta|+|\zeta^*|^2|\zeta_1|^{\frac{m}{2}+1}+|\zeta_1|^{m+1}),
\end{split}
\end{equation}
where $\zeta^*=(0,\zeta_2,\cdots,\zeta_{n-1},0)$.
\end{proposition}

\begin{nx} The coordinate changes as above are unique and hence the map $\Phi_{z'}$ is defined uniquely.  
\end{nx}

We now show how to define the polydisc around $z'$. Set 
\begin{equation}\label{Eq5}
\begin{split} 
A_l(z')&=\max \{|a_{j,k}(z')|, \ j+k=l\}\ (  2\leq l \leq m),\\
B_{l'}(z')&=\max \{|b^\alpha_{j,k}(z')|, \ j+k=l',\ 2\leq \alpha\leq n-1\}\ (  2\leq l' \leq \frac{m}{2}).
\end{split}
\end{equation}
For each $\delta>0$, we define $\tau(z',\delta)$ as follows 
\begin{equation}\label{Eq6} 
\tau(z',\delta)=\min \{\big( \delta/A_l(z') \big)^{1/l},\ \big( \delta^{\frac{1}{2}}/B_{l'}(z') \big)^{1/{l'}},\ 2\leq l \leq m,\ 2\leq l' \leq \frac{m}{2} \}.
\end{equation}

Since the type of $\partial \Omega$ at $\xi_0$ equals $m$ and the Levi form has rank at least $n-2$ at $\xi_0$, $A_m(\xi_0)\ne 0$. Hence if $U$ is sufficiently small, then $|A_m(z')|\geq c>0$ for all $z'\in U$. This gives  the inequality
\begin{equation}\label{Eq7} 
\delta^{1/2}\lesssim \tau(z',\delta)  \lesssim \delta^{1/m}\  (z'\in U).
\end{equation}

The definition of $\tau(z',\delta)$ easily implies that if $\delta'<\delta''$, then  
  \begin{equation}\label{Eq8} 
(\delta'/\delta'')^{1/2}\tau(z',\delta'')  \leq \tau(z',\delta') \leq  (\delta'/\delta'')^{1/m}\tau(z',\delta'').
\end{equation}

Now set $\tau_1(z',\delta)=\tau(z',\delta)=\tau,\ \tau_2(z',\delta)=\cdots=\tau_{n-1}(z',\delta)=\delta^{1/2}$, $\tau_n(z',\delta)=\delta$ and define
\begin{equation}\label{Eq9} 
R(z',\delta)=\{\zeta\in \CC^n: |\zeta_k|<\tau_k(z',\delta), k=1,\cdots,n\}
\end{equation}
and 
\begin{equation}\label{Eq10} 
Q(z',\delta)=\{\Phi^{-1}_{z'}(\zeta):\zeta\in R(z',\delta)\}.
\end{equation}

In the sequal we denote $D_k^l$ any partial derivative operator of the form $\frac{\partial }{\partial \zeta^\mu_k}\frac{\partial }{\partial \bar\zeta^\nu_k}$, where $\mu+\nu=l,\ k=1,2,\cdots,n$.

In order to prove the homogeneous property of $Q(z',\delta)$ we need two lemmas.
\begin{Lemma} \cite[Prop. 2.3, p. 807]{Cho1}. \label{L32}
Let $z'$ be an arbitrary point in $U$. Then the function $\rho (\zeta)=r(\Phi^{-1}_{z'}(\zeta))$ satisfies
\begin{equation}\label{Eq11} 
\begin{split}
&|\rho(\zeta)-\rho(0)|\lesssim \delta \\
&|D^i_k D_1^l \rho(\zeta)|\lesssim \delta \tau_1(z',\delta)^{-l}\tau_k(z',\delta)^{-i} ,
\end{split}
\end{equation}
 for $\zeta\in R(z',\delta)$ and $ l+\frac{im}{2}\leq m, \ i=0,1;\  k=2,\cdots ,n-1$.
\end{Lemma}

\begin{Lemma} \cite[ Cor. 2.8, p. 812 ]{Cho1}.\label{L33} Suppose that $z\in Q(z',\delta)$. Then 
\begin{equation}\label{Eq12} 
\tau(z,\delta)\approx \tau(z',\delta).
\end{equation}
\end{Lemma}

We now apply Lemma \ref{L33} to the question of how the polydiscs $Q(z',\delta)$ and $Q(z'',\delta)$ are related. 
Let $\Phi^{-1}_{z'}$ be the map associated with $z'$ as in Proposition \ref{P31}. Define $\zeta''$ by $z''=\Phi^{-1}_{z'}(\zeta'').$ 
Applying Proposition \ref{P31} at the point $\zeta''$ with $r$ replaced by $\rho=r\circ\Phi^{-1}_{z'},$ 
we obtain a map $\Phi^{-1}_{\zeta''}: \CC^n\to\CC^n$ defined by $\Phi^{-1}_{\zeta''}=\varphi_1\circ\varphi_2\circ\varphi_3\circ\varphi_4\circ\varphi_5$ where

 $z=\varphi_1(u)$ defined by
\begin{equation*}
\begin{split}
z_n&={\zeta''}_n+u_n+ \sum_{j=1}^{n-1} b_j u_j,\\
 z_j&={\zeta''}_j+u_j,\ j=1,\cdots, n-1,
\end{split}
\end{equation*}
 $u=\varphi_2(v)$ defined by 
\begin{equation*}
\begin{split}
u_1&=v_1, \ u_n=v_n,\quad \text{and}\\
 u_j&=\sum_{k=2}^{n-1}\bar P_{jk} v_k,\ j=2,\cdots, n-1,
\end{split}
\end{equation*}
$v=\varphi_3(w)$ defined by 
\begin{equation*}
\begin{split}
v_1&=w_1, \ v_n=w_n,\quad \text{and}\\
 v_j&=\lambda_j w_j,\ j=2,\cdots, n-1,
\end{split}
\end{equation*}
 $w=\varphi_4(t)$ defined by
\begin{equation*}
\begin{split}
w_n&=t_n+\sum_{2\leq k\leq m}d_k t_1^k +\sum_{\alpha=2}^{n-1}\sum_{1\leq k\leq m/2} d_{\alpha, k} t_\alpha t_1^k+\sum_{\alpha=2}^{n-1} c_\alpha t_\alpha^2, \\
 w_j&=t_j,\ j=1,\cdots, n-1,
\end{split}
\end{equation*}
and $t=\varphi_5(\xi)$ defined by
\begin{equation*}
\begin{split}
t_1&=\xi_1,\ t_n=\xi_n,\\
 t_\alpha &=\xi_\alpha+\sum_{1\leq k\leq \frac{m}{2}} e_{\alpha, k} \xi_1^k,\ \alpha=2,\cdots, n-1.
\end{split}
\end{equation*}
\begin{equation*} 
\begin{split}
\rho(\Phi_{\zeta''}^{-1}(\xi))&=\rho(\zeta'')+Re \xi_n+ \sum_{\substack{j+k\leq m\\
 j,k>0}} a_{jk}(\zeta'')\xi_1^j \bar \xi_1^k\\
&+\sum_{\alpha=2}^{n-1}|\xi_\alpha|^2+ \sum_{\alpha=2}^{n-1}Re (( \sum_{\substack{j+k\leq \frac{m}{2}\\
 j,k>0}} b^\alpha_{jk}(\zeta'')\xi_1^j \bar \xi_1^k)\xi_\alpha)\\
&+O(|\xi_n| |\xi|+|\xi''|^2|\xi|+|\xi''|^2|\xi_1|^{\frac{m}{2}+1}+|\xi_1|^{m+1}),
\end{split}
\end{equation*}

Since the composition $\Phi^{-1}_{z'}\circ \Phi^{-1}_{\zeta''}$ gives a map of the same  form as $\Phi^{-1}_{z''}$, where $\Phi^{-1}_{z''}$ is obtained by applying Proposition \ref{P31} to the function $r$ and $z''$, we conclude from the uniqueness statement in Proposition \ref{P31} that
\begin{equation}\label{E13}
\Phi^{-1}_{z''}=\Phi^{-1}_{z'}\circ \Phi^{-1}_{\zeta''}.
\end{equation}

In order to study $Q(z'',\delta)$ we must therefore examine the map $\Phi^{-1}_{\zeta''}$.
\begin{Lemma}\label{L34} Suppose that $z''\in Q(z',\delta)$. Then 
\begin{equation}\label{Eq14} 
\begin{split}
&|b_j|\lesssim \delta \tau_j(z',\delta)^{-1},\ |c_\alpha|\lesssim \delta \tau_\alpha(z',\delta)^{-2},\ |d_k|\lesssim \delta \tau_1(z',\delta)^{-k},  \\
&|d_{\alpha,k}|\lesssim \delta \tau_1(z',\delta)^{-l}\tau_\alpha(z',\delta)^{-1}, \ | e_{\alpha,l}|\lesssim \delta \tau_1(z',\delta)^{-l}\tau_\alpha(z',\delta)^{-1},
\end{split}
\end{equation}
for  $1\leq j\leq n-1,\ 1\leq k\leq m, \ 2\leq \alpha \leq n-1,\  1\leq l\leq \frac{m}{2}$
\end{Lemma}

\begin{proof}
From the proof of Proposition \ref{P31}, we see that 
\begin{equation*}
\begin{split}
&b_j=-(\frac{\partial \rho}{\partial \zeta_1})^{-1}\frac{\partial \rho(\zeta'') }{\partial \zeta_j},\\
&c_\alpha=-\frac{\partial^2 \rho (0)}{\partial \zeta^2_\alpha},\\
&d_k=-\frac{2}{k!} \frac{\partial^k \rho(0) }{\partial w_1^k}, \\
&d_{\alpha,l}=-\frac{2}{(l+1)!} \frac{\partial^{l+1} \rho(0) }{\partial w_\alpha \partial w_1^l},\\
&e_{\alpha,l}=-\frac{1}{(l+1)!} \frac{\partial^{l+1} \rho(0) }{\partial \bar t_\alpha\partial t_1^l},
\end{split}
\end{equation*}
for $1\leq j\leq n-1,\ 1\leq k\leq m, \ 2\leq \alpha \leq n-1,\  1\leq l\leq \frac{m}{2}$. By Lemma \ref{L32} and the definition of the biholomorphism $\Phi^{-1}_{\zeta''}$ we conclude that (\ref{Eq14}) holds. 
\end{proof}
\begin{proposition}\label{P2} There exists a constant $C$ such that if $z''\in Q(z',\delta)$, then
\begin{equation}\label{Eq15}
Q(z'',\delta)\subset Q(z',C\delta)
\end{equation}
and
\begin{equation}\label{Eq16}
Q(z',\delta)\subset Q(z'',C\delta)
\end{equation}
\end{proposition}
\begin{proof}
Define $S(z'',\delta)=\{\Phi^{-1}_{\zeta''}(\xi): \xi\in R(z'',\delta)\}$. It easy to see that $Q(z'',\delta)=\Phi^{-1}_{z'}\circ S(z'',\delta)$. Thus, in order to prove (\ref{Eq15}) it suffices to show that 
\begin{equation}\label{Eq17}
S(z'',\delta)\subset R(z',C\delta)
\end{equation}

Indeed, for each $\xi\in R(z'',\delta)$, set $t=\varphi_5(\xi)$. By Lemma \ref{L33} and Lemma \ref{L34}, we have
\begin{equation*}
\begin{split}
 |t_1|&=|\xi_1|\leq\tau_1(z'',\delta)\lesssim\tau_1(z',\delta),\\
|t_n|&=|\xi_n|\leq \tau_n(z'',\delta)= \tau_n(z',\delta)=\delta,\\
|t_\alpha|&\leq |\xi_\alpha|+\sum_{k=2}^{n-1}|e_{\alpha, k}| |\xi_1|^k\lesssim\tau_\alpha(z'',\delta)+ \delta \tau_1(z',\delta)^{-k}\tau_\alpha(z',\delta)^{-1}\tau_1(z'',\delta)^k\\
&\lesssim \tau_\alpha(z',\delta),\quad 2\leq \alpha\leq n-1.
\end{split}
\end{equation*}
We also set $w=\varphi_4(t)$, by Lemma \ref{L34}, we have
\begin{equation*}
\begin{split}
|w_n|&\leq |t_n|+\sum_{k=2}^{ m} |d_k| |t_1|^k+\sum_{\alpha=2}^{n-1}\sum_{k=1}^{m/2}|d_{\alpha,k}| |t_\alpha| |t_1|^k+\sum_{\alpha=2}^{n-1}|c_\alpha| |t_\alpha|^2\\
       &\lesssim \tau_n(z',\delta)+\sum_{k=2}^{ m} \delta  \tau_1(z',\delta)^{-k}  \tau_1(z',\delta)^k+\sum_{\alpha=2}^{ n-1} \delta  \tau_\alpha(z',\delta)^{-2}  \tau_\alpha(z',\delta)^2\\
&+\sum_{\alpha=2}^{n-1}\sum_{k=1}^{m/2} \delta  \tau_1(z',\delta)^{-k}\tau_\alpha(z',\delta)^{-1}\tau_\alpha(z',\delta)\tau_1(z',\delta)^k\lesssim \delta=\tau_n(z',\delta) ,\\
|w_j|&=|t_j|\lesssim \tau_j(z',\delta), \quad 1\leq j\leq n-1.
\end{split}
\end{equation*}

Set $v=\varphi_3(w)$, $u=\varphi_2(v)$ and $\zeta=\varphi_1(u)$. It is easy to see that $|v_j|\lesssim \tau_j(z',\delta),\ |u_j|\lesssim \tau_j(z',\delta),\ |\zeta_j|\lesssim \tau_j(z',\delta),\ 1\leq j\leq n$ and hence, (\ref{Eq17}) holds if $C$ is sufficiently large.

To prove (\ref{Eq16}), define  $P(z',\delta)=\{\Phi_{\zeta''}(\zeta):\zeta \in R(z',\delta)\}$, it easy to see that $Q(z',\delta)=\Phi^{-1}_{z''}\circ P(z'',\delta)$. Thus, it suffices to show that
\begin{equation}\label{Eq18}
P(z',\delta)\subset R(z'',C\delta).
\end{equation}

Indeed, we see that $\Phi_{\zeta''}=\varphi_5^{-1}\circ\varphi_4^{-1}\circ\varphi_3^{-1}\circ\varphi_2^{-1}\circ\varphi_1^{-1}$ and  $$\tau(z',\delta)\lesssim \tau(z'',\delta).$$

Applying (\ref{Eq14}) in the same way as above, we conclude that if $\zeta\in R(z',\delta)$, then $\xi=\Phi_{\zeta''}(\zeta)\in R(z'',C \delta)$, where $C$ is sufficiently large.

Hence, (\ref{Eq18}) holds. The proof is completed.
\end{proof}
\subsection{Dilation of coordinates}

Let $\Omega$ be a domain in $\CC^n$. Suppose that $\partial \Omega $ is pseudoconvex,  of finite type and is smooth of class $C^\infty$ near a boundary point $\xi_0\in \partial \Omega$ and suppose that the Levi form has rank at least $n-2$ at $\xi_0$. 

We may assume that $\xi_0=0$ and the rank of Levi form at $\xi_0$ is exactly $n-2$. Let $\rho$ be a smooth defining function for  $\Omega $. After a linear change of coordinates, we can find coordinate functions $z_1,\cdots, z_n$ defined on a neighborhood $U_0$ of $\xi_0$ such that 
\begin{equation*}
\begin{split}
\rho(z)&=Re z_n+ \sum_{\substack{j+k\leq m\\
 j,k>0}} a_{j,k}z_1^j \bar z_1^k\\
&+\sum_{\alpha=2}^{n-1}|z_\alpha|^2+ \sum_{\alpha=2}^{n-1} \sum_{\substack{j+k\leq \frac{m}{2}\\
 j,k>0}}Re (( b^\alpha_{j,k}z_1^j \bar z_1^k)z_\alpha)\\
&+O(|z_n| |z|+|z^*|^2|z|+|z^*|^2|z_1|^{\frac{m}{2}+1}+|z_1|^{m+1}),
\end{split}
\end{equation*}
where $z^*=(0,z_2,\cdots,z_{n-1},0)$.

By Proposition \ref{P31}, for each point $\eta$ in a small neighborhood of the origin, there exists a unique automorphism $\Phi_\eta$ of $\CC^n$ such that
\begin{equation}\label{Eq19} 
\begin{split}
\rho(\Phi_{\eta}^{-1}(w))-\rho(\eta)&= Re w_n+ \sum_{\substack{j+k\leq m\\
 j,k>0}} a_{j,k}(\eta)w_1^j \bar w_1^k\\
&+\sum_{\alpha=2}^{n-1}|w_\alpha|^2+ \sum_{\alpha=2}^{n-1} \sum_{\substack{j+k\leq \frac{m}{2}\\
 j,k>0}} Re [(b^\alpha_{j,k}(\eta)w_1^j \bar w_1^k)w_\alpha]\\
&+O(|w_n| |w|+|w^*|^2|w|+|w^*|^2|w_1|^{\frac{m}{2}+1}+|w_1|^{m+1}),
\end{split}
\end{equation}
where $w^*=(0,w_2,\cdots,w_{n-1},0)$.

We define an anisotropic dilation $\Delta_\eta^\epsilon$ by $$\Delta_\eta^\epsilon (w_1,\cdots,w_n)=(\frac{w_1}{\tau_1(\eta,\epsilon)},\cdots,\frac{w_n}{\tau_n(\eta,\epsilon)}),$$
where $\tau_1(\eta,\epsilon)=\tau(\eta,\epsilon),\  \tau_k(\eta,\epsilon)=\sqrt{\epsilon}\ (2\leq k\leq n-1), \ \tau_n(\eta,\epsilon)=\epsilon$. 

For each $\eta\in \partial \Omega$, if we set $\rho_\eta^\epsilon(w)=\epsilon^{-1}\rho\circ \Phi_\eta^{-1}\circ(\Delta_\eta^\epsilon)^{-1}(w)$, then 
\begin{equation}\label{Eq20} 
\begin{split}
\rho_\eta^\epsilon(w)&= Re w_n+ \sum_{\substack{j+k\leq m\\
 j,k>0}} a_{j,k}(\eta) \epsilon^{-1} \tau(\eta,\epsilon)^{j+k}w_1^j \bar w_1^k+\sum_{\alpha=2}^{n-1}|w_\alpha|^2\\
&+ \sum_{\alpha=2}^{n-1} \sum_{\substack{j+k\leq \frac{m}{2}\\
 j,k>0}}Re ( b^\alpha_{j,k}(\eta)\epsilon^{-1/2} \tau(\eta,\epsilon)^{j+k}w_1^j \bar w_1^kw_\alpha)+O(\tau(\eta,\epsilon)),
\end{split}
\end{equation}

For each $\eta \in U_0$, we define pseudo-balls $Q(\eta,\epsilon)$ by
\begin{equation}\label{Eq21}
\begin{split}
Q(\eta,\epsilon)&:=\Phi^{-1}_\eta(\Delta_\eta^\epsilon)^{-1}(D\times\cdots\times D)\\
&=\Phi^{-1}_\eta\{ |w_k|<\tau_k(\eta,\epsilon), \ 1\leq k\leq n\},
\end{split}
\end{equation}
where $D_r:=\{z\in \CC: |z|<r\}$. There exist constants $0\leq \alpha\leq 1$ and $C_1,C_2,C_3\geq 1$ such that for $\eta,\eta'\in U_0$ and $\epsilon \in (0,\alpha]$ the following estimates are satisfied with $\eta\in Q(\eta',\epsilon)$
\begin{equation}\label{Eq22}
\rho(\eta)\leq \rho(\eta')+C_1\epsilon,
\end{equation}
\begin{equation}\label{Eq23}
\frac{1}{C_2}\tau(\eta,\epsilon)\leq\tau(\eta',\epsilon)\leq C_2\tau(\eta,\epsilon),
\end{equation}
\begin{equation}\label{Eq24}
Q(\eta,\epsilon)\subset Q(\eta',C_3\epsilon)\ \text{and}\ Q(\eta',\epsilon)\subset Q(\eta,C_3\epsilon).
\end{equation}
Set $\epsilon (\eta):=|\rho(\eta)|, \ \Delta_\eta:=\Delta_\eta^{\epsilon (\eta)}$ and $C_4=C_1+1$. By (\ref{Eq22}), we have 
\begin{equation}\label{Eq25}
\eta \in Q(\eta',\epsilon(\eta'))\Rightarrow \epsilon(\eta)\leq C_4\epsilon(\eta').
\end{equation}

Fix neighborhoods $W_0,V_0$ of the origin with $W_0\subset V_0\subset U_0.$ Then for sufficiently small constants $\alpha_1,\alpha_0 \ (0<\alpha_1\leq \alpha_0<1),$ we have 
\begin{equation}\label{Eq26}
\eta \in V_0\ \text{ and}\ 0<\epsilon\leq \alpha_0\Rightarrow Q(\eta,\epsilon)\subset U_0 \ \text{and}\ \epsilon(\eta)\leq \alpha_0
\end{equation}
\begin{equation}\label{Eq27}
\eta \in W_0\ \text{ and}\ 0<\epsilon\leq \alpha_1\Rightarrow Q(\eta,\epsilon)\subset V_0 .
\end{equation}

Define a pseudo-metric by $$M(\eta,\overrightarrow{X}):=\sum_{k=1}^n \frac{|({\Phi'}_\eta(\eta)\overrightarrow{X})_k|}{\tau_k(\eta,\epsilon(\eta))} =\|\Delta_\eta\circ {\Phi'}_\eta(\eta)\overrightarrow{X}\|_1$$ on $U_0$. By (\ref{Eq7}), one has $$\frac{\|\overrightarrow{X}\|_1}{\epsilon(\eta)^{1/m}}\lesssim M(\eta,\overrightarrow{X})\lesssim        \frac{\|\overrightarrow{X}\|_1}{\epsilon(\eta)}. $$
\begin{Lemma}\label{L4}
There exist constants $K\geq 1 (K=C_3\cdot C_4)$ and $0<A< 1$ such that for each integer $N\geq 1$ and each holomorphic $f: D_N\to U_0$ satisfies $M(f(u),f'(u))\leq A$ on $D_N$, we have 
$$f(0)\in W_0\ \text{and}\ K^{N-1}\epsilon(f(0))\leq \alpha_1\Rightarrow \overline{ f(D_N)}\subset Q[f(0),K^N\epsilon(f(0))].$$
\end{Lemma}

\begin{proof}
Let $\eta_0\in V_0$ and $\eta\in Q(\eta_0,\epsilon_0)$, where $\epsilon_0=\epsilon(\eta_0)$. From (\ref{Eq25}), (\ref{Eq23}) and (\ref{Eq8}) one has $\epsilon (\eta)\leq C_4\epsilon_0$ and 
$$\tau(\eta,\epsilon(\eta))\leq\tau(\eta,C_4\epsilon_0)\leq C_2\sqrt{C_4} \tau(\eta_0,\epsilon_0).$$  

Thus $M(\eta,\overrightarrow{X})\gtrsim \sum\limits_{k=1}^n \dfrac{|({\Phi'}_\eta(\eta)\overrightarrow{X})_k|}{\tau_k(\eta_0,\epsilon_0)}.$

In order to replace ${\Phi'}_\eta(\eta)$ by ${\Phi'}_{\eta_0}(\eta)$ in this inequality, we consider the automorphism $\Psi:={\Phi}_\eta\circ {\Phi}_{\eta_0}^{-1}$ which equals ${\Phi}^{-1}_a=\varphi_1\circ\varphi_2\circ\varphi_3\circ\varphi_4\circ\varphi_5$ where $a:={\Phi}_\eta(\eta_0) $ and $\varphi_j (1\leq j\leq 5)$ are given in the previous section.

If we set $\Lambda:={\Phi'}_\eta(\eta)\circ({\Phi'}_{\eta_0}(\eta))^{-1}=\Psi'({\Phi}_{\eta_0}(\eta))$, then $\Lambda={\varphi'}_1\circ{\varphi'}_2\circ{\varphi'}_3\circ{\varphi'}_4\circ{\varphi'}_5$. By a simple computation, we have 
$${\varphi'}_1(w_1,\cdots,w_n)=(w_1, w_2,\cdots, w_n +\sum_{k=1}^{n-1} b_k w_k)$$ 
where $|b_k|\leq C.\frac{\epsilon_0} {\tau_k(\eta_0,\epsilon_0)} \ (1\leq k\leq n-1)$ for some constant $C\geq 1$.

Set $\overrightarrow{Y}:={\Phi'}_{\eta_0}(\eta)\overrightarrow{X},\ \overrightarrow{Y^4}:={\varphi'}_5\overrightarrow{Y},\ \overrightarrow{Y^3}:={\varphi'}_4\overrightarrow{Y^4},\ \overrightarrow{Y^2}:={\varphi'}_3\overrightarrow{Y^3}$ and $\overrightarrow{Y^1}:={\varphi'}_2\overrightarrow{Y^2} $, since ${\Phi'}_{\eta}(\eta)\overrightarrow{X}=\Lambda [\overrightarrow{Y}]={\varphi'}_1\overrightarrow{Y^1}$, we have
\begin{equation*}
\begin{split}
M(\eta,\overrightarrow{X})&\gtrsim \frac{|({\Phi'}_\eta(\eta)\overrightarrow{X})_1|}{\tau_{1}(\eta_0,\epsilon_0)}+\cdots+\frac{|({\Phi'}_\eta(\eta)\overrightarrow{X})_2|}{\tau_{n-1}(\eta_0,\epsilon_0)}+\frac{|({\Phi'}_\eta(\eta)\overrightarrow{X})_{n}|}{2 C \epsilon_0}\\
&\gtrsim \sum_{k=1}^{n-1} (1-\frac{|b_k|\tau_k(\eta_0,\epsilon_0)}{2 C\epsilon_0})\frac{|Y^1_k|}{\tau_k(\eta_0,\epsilon_0)}+\frac{|Y^1_n|}{2 C\epsilon_0}\\
&\gtrsim \sum_{k=1}^{n}  \frac{|Y^1_k|}{\tau_k(\eta_0,\epsilon_0)}.
\end{split}
\end{equation*}

Because of the definition of the maps $\varphi_2$ and $\varphi_3$, it is easy to show that
$$\sum_{k=1}^{n}  \frac{|Y^1_k|}{\tau_k(\eta_0,\epsilon_0)}\gtrsim \sum_{k=1}^{n}  \frac{|Y^2_k|}{\tau_k(\eta_0,\epsilon_0)}\gtrsim \sum_{k=1}^{n}  \frac{|Y^3_k|}{\tau_k(\eta_0,\epsilon_0)}.$$
 
Next we also have
 $${\varphi'}_4(w_1,\cdots,w_n)=(w_1, w_2,\cdots, w_n+\sum_{k=1}^{n-1} \gamma_k w_k )$$ 
where 
\begin{equation*}
\begin{split}
|\gamma_k|&\lesssim\sum_{j=1}^{m/2}|d_{k,j}| \tau_1(\eta_0,\epsilon_0)^j +2.|c_k|\tau_k(\eta_0,\epsilon_0)\leq C.\frac{\epsilon_0} {\tau_k(\eta_0,\epsilon_0)},\\
|\gamma_1|&\lesssim \sum_{\alpha=2}^{n-1}\sum_{j=1}^{m/2}|d_{\alpha,j}| \tau_\alpha(\eta_0,\epsilon_0) .j. \tau_1(\eta_0,\epsilon_0)^{j-1}+\sum_{j=2}^{m}|d_{j}| .j. \tau_1(\eta_0,\epsilon_0)^{j-1} \\
&\leq C.\frac{\epsilon_0} {\tau_1(\eta_0,\epsilon_0)} .
\end{split}
\end{equation*}
for $ k=2,\cdots, n-1$ and  some constant $C\geq 1$.
Using the same argument as above we have  
$$\sum_{k=1}^{n}  \frac{|Y^3_k|}{\tau_k(\eta_0,\epsilon_0)}\gtrsim \sum_{k=1}^{n}  \frac{|Y^4_k|}{\tau_k(\eta_0,\epsilon_0)}.$$

The derivative of $\varphi_5$ is defined by 
$${\varphi'}_5(w_1,\cdots,w_n)=(w_1, w_2+\beta_2 w_1,\cdots,w_{n-1}+\beta_{n-1} w_1 ,w_n )$$ 
where $|\beta_k |\lesssim \sum\limits_{l=1}^{\frac{m}{2}}|e_{k,l}| .l. \tau_1(\eta_0,\epsilon_0)^{l-1} \leq C.\frac{\epsilon_0} {\tau_k (\eta_0,\epsilon_0)\tau_1 (\eta_0,\epsilon_0)} \ (2\leq k \leq n-1)$  for some constant $C\geq 1$.

Since $\overrightarrow{Y^4}={\varphi'}_5\overrightarrow{Y}$, we have
\begin{equation*}
\begin{split}
\sum_{k=1}^{n}  \frac{|Y^4_k|}{\tau_k(\eta_0,\epsilon_0)}&\gtrsim \frac{|Y_1^4|}{\tau_1(\eta_0,\epsilon_0)}+\sum_{k=2}^{n-1}\frac{|Y^4_k|}{2nC\tau_k(\eta_0,\epsilon_0)}+\frac{|Y^4_n|}{\tau_n(\eta_0,\epsilon_0)}\\
&\gtrsim  (1-\sum_{k=2}^{n-1}\frac{|\beta_k|\tau_1(\eta_0,\epsilon_0)}{2n C\tau_k(\eta_0,\epsilon_0)})\frac{|Y_1|}{\tau_1(\eta_0,\epsilon_0)}+\\
&+\sum_{k=2}^{n-1}\frac{|Y_k|}{2nC\tau_k(\eta_0,\epsilon_0)}+\frac{|Y_n|}{\tau_n(\eta_0,\epsilon_0)}\\
 &\gtrsim \sum_{k=1}^{n}  \frac{|Y_k|}{\tau_k(\eta_0,\epsilon_0)}=\sum_{k=1}^{n}  \frac{|({\Phi'}_{\eta_0}(\eta)\overrightarrow{X})_k|}{\tau_k(\eta_0,\epsilon_0)}.
\end{split}
\end{equation*}

Therefore, there exists a constant $1\geq A>0$ such that $M(\eta,\overrightarrow{X})\geq A\|\Delta_{\eta_0}\circ{\Phi'}_{\eta_0}(\eta)\overrightarrow{X}\|_1$ for every $\eta_0\in V_0$ and for  every $ \eta\in Q(\eta_0,\epsilon(\eta_0))$. By this observation, we can finish the proof.

a) If $N=1$, the inclusion $f(D_1)\subset Q(\eta_0,\epsilon_0)$ is satisfied as $f(0)\in W_0$. This deduces immediately from the observation that $\|\frac{d}{du}\Delta_{\eta_0}\circ{\Phi}_{\eta_0}\circ f(u)\|_1\leq 1$ as $f(u)\in  Q(\eta_0,\epsilon_0).$

b) Suppose now $N\geq 2$ and $f(0)\in W_0$. Fix $\theta_0\in (0,2\pi]$ and let $u_j=j e^{i\theta_0},\ \eta_j:=f(u_j)$ and $\epsilon_j=\epsilon(\eta_j)$. It is sufficient to show that $\overline{f[D(u_i,1)]}\subset  Q(\eta_0,K^i\epsilon_0) $ for $i\leq N-1$.

For $i=1$, this assertion is proved in a). Suppose that these inclutions are satisfied for $i\leq j<N-1$. Since $\eta_{j+1}\in  Q(\eta_0,K^j \epsilon_0)$, we have $\epsilon_{j+1}\leq C_4 K^j \epsilon_0<\alpha_1.$ Moreover,  since $\eta_0\in W_0,$ it implies that $\eta_{j+1}\in  V_0$ (see (\ref{Eq27})). We may apply a) to the restriction of $f$ to $D(u_{j+1},1)$
\begin{equation*}
\begin{split}
\overline{f[D(u_{j+1},1)]}&\subset  Q(\eta_{j+1},\epsilon_{j+1})\subset Q(\eta_{j+1},C_4 K^j\epsilon_0)\\
&\subset  Q(\eta_0,C_3 C_4 K^j\epsilon_0)=Q(\eta_0, K^{j+1}\epsilon_0).
  \end{split}
\end{equation*}
\end{proof}

For any sequence $\{\eta_p\}_p$ of points tending to the origin in $U_0\cap\{\rho<0\}=:U_0^-$, we associate with a sequence of points ${\eta'}_p=(\eta_{1p}, \cdots, \eta_{np}+\epsilon_p)$, $ \epsilon_p>0$, ${\eta'}_{p}$ in the hypersurface $\{\rho=0\}$. Consider the sequence of dilations $\Delta_{{\eta'}_{p}}^{\epsilon_p}$. Then $\Delta_{{\eta'}_{p}}^{\epsilon_p}\circ \Phi_{{\eta'}_p}({\eta}_p)=(0,\cdots,0,-1)$. By (\ref{Eq20}), we see that $\Delta_{{\eta'}_{p}}^{\epsilon_p}\circ \Phi_{{\eta'}_p}(\{\rho=0\}) $ is defined by an equation of the form
\begin{equation}\label{Eq28}
\begin{split}
Re w_n+ P_{{\eta'}_p}(w_1,\bar w_1)+\sum_{\alpha=2}^{n-1}|w_\alpha|^2+ \sum_{\alpha=2}^{n-1}Re &(Q^\alpha_{{\eta'}_p}(w_1,\bar w_1)w_\alpha)+\\
&+O(\tau({\eta'}_p,\epsilon_p))=0,
\end{split}
\end{equation}
where
\begin{equation*}
\begin{split}
&P_{{\eta'}_p}(w_1,\bar w_1):=\sum_{\substack{j+k\leq m\\
 j,k>0}} a_{j,k}({\eta'}_p) \epsilon_p^{-1} \tau({\eta'}_p,\epsilon_p)^{j+k}w_1^j \bar w_1^k,\\
&Q^\alpha_{{\eta'}_p}(w_1,\bar w_1):= \sum_{\substack{j+k\leq \frac{m}{2}\\
 j,k>0}} b^\alpha_{j,k}({\eta'}_p)\epsilon_p^{-1/2} \tau({\eta'}_p,\epsilon_p)^{j+k}w_1^j \bar w_1^k.
\end{split}
\end{equation*}

Note that from (\ref{Eq5}) we know that the coefficients of $P_{{\eta'}_p}$ and $Q^\alpha_{{\eta'}_p}$ are bounded by one. But the polynomials $Q^\alpha_{{\eta'}_p}$ are less important than $P_{{\eta'}_p}$. In \cite{Cho1}, S. Cho proved the following lemma.
\begin{Lemma} \cite[Lem. 2.4, p. 810]{Cho1}.\label{L5} $|Q^\alpha_{{\eta'}_p}(w_1,\bar w_1)|\leq \tau({\eta'}_p,\epsilon_p)^{\frac{1}{10}} $ for all $\alpha=2,\cdots, n-1$ and $|w_1|\leq 1$.
\end{Lemma}
By Lemma \ref{L5}, it follows that after taking a subsequence, $\Delta_{{\eta'}_{p}}^{\epsilon_p}\circ \Phi_{{\eta'}_p}(U_0^-)$ converges to the following domain
\begin{equation}\label{Eq29} 
M_P:=\{\hat\rho:=Re w_n+ P(w_1,\bar w_1)+\sum_{\alpha=2}^{n-1}|w_\alpha|^2<0
\}.
\end{equation}
where $P(w_1,\bar w_1)$ is a polynomial of degree $\leq m$ without harmonic terms.
 
Since $M_P$ is a smooth limit of the pseudoconvex domains $ \Delta_{{\eta'}_{p}}^{\epsilon_p}\circ \Phi_{{\eta'}_p}(U_0^-) $, it is pseudoconvex. Thus the function $\hat\rho$ in (\ref{Eq29}) is plurisubharmonic, and hence $P$ is a subharmonic polynomial whose Laplacian does not vanish identically.
\begin{Lemma} The domain $M_P$ is Brody hyperbolic.
\end{Lemma}
\begin{proof} If $\varphi:\CC\to M_P$ is holomorphic, then the subharmonic functions $Re \varphi_n+ P\circ \varphi_1+\sum_{\alpha=2}^{n-1}|\varphi_\alpha|^2$ and $Re \varphi_n+ P\circ \varphi_1$ are negative on $\CC$. Consequently, they are constant. This implies that $P\circ \varphi_1$ is harmonic. Hence $\varphi_1, Re\varphi_n$ and $\varphi_n$ are constant. In addition, the function $\sum_{\alpha=2}^{n-1}|\varphi_\alpha|^2$ is also constant and hence $\varphi_\alpha \ (2\leq \alpha\leq n-1)$ are constant.
\end{proof}

\subsection{Estimates of Kobayashi metric}

Recall that the Kobayashi metric $K_\Omega$ of $\Omega$ is defined by
$$K_\Omega(\eta,\overrightarrow{X}):=\inf\{\frac{1}{R}|\ \exists f: D\to \Omega\ \text{such that} f(0)=\eta, f'(0)=R \overrightarrow{X}\}.$$

By the same argument as in \cite{Ber} page 93, there exists a neighborhood $U$ of the origin with $U\subset U_0$ such that 
$$K_\Omega(\eta,\overrightarrow{X})\leq K_{\Omega\cap U_0}(\eta,\overrightarrow{X})\leq  2K_\Omega(\eta,\overrightarrow{X})\ \text{for all}\ \eta\in U\cap\Omega.$$

We need the following lemma (see \cite{Ber3}) 

\begin{Lemma}\label{L6} Let $(X,d)$ be a complete metric space and let $M: X\to \RR^+$ be a locally bounded function. Then, for all $\sigma>0$ and for all $u\in X$ satisfying  $M(u)>0$, there exists $v\in X$ such that

\begin{enumerate}
\item[(i)] $d(u,v )\leq \frac{2}{\sigma M(u)}$
\item[(ii)] $M(v)\geq M(u)$
\item[(iii)] $M(x)\leq 2 M(v)$ if $d(x,v)\leq \frac{1}{\sigma M(v)}$.
\end{enumerate}
\end{Lemma}

\begin{proof}
If $v$ does not exist, one contructs a sequence $(v_j)$ such that $v_0=u$, $M(v_{n+1})\geq 2 M(v_j)\geq 2^{n+1}M(u)$ and $d(v_{n+1},v_j)\leq \frac{1}{\sigma M(v_j)}\leq \frac{1}{\sigma M(u) 2^n}$. This sequence  is Cauchy. 
\end{proof}

\begin{theorem}\label{T2}Let $\Omega$ be a domain in $\CC^n.$  Suppose that $\partial \Omega$ is pseudoconvex, of finite type and is smooth of class $C^\infty$ near a boundary point $p\in\partial \Omega$ and suppose that the Levi form has rank at least $n-2$ at $\xi_0$. Then, there exists a neighborhood $V$ of $\xi_0$ such that

$$M(\eta,\overrightarrow{X})\lesssim K_{\Omega}(\eta,\overrightarrow{X})\lesssim  M(\eta,\overrightarrow{X})\ \text{for all} \ \eta\in V\cap\Omega.$$
\end{theorem}
\begin{proof}[Proof of Theorem \ref{T2}]
The second inequality is obvious, by the definition. We are going to prove the first inequality. We may also assume that $\xi_0=(0,\cdots,0)$. It  suffices  to show that for $\eta$ near $0$ and $\overrightarrow{X}$ is not zero, we have
$$K_\Omega\biggl(\eta, \dfrac{\overrightarrow{X}}{M(\eta,\overrightarrow{X})}\biggl)\gtrsim 1.$$

Suppose that this is not true. Then there exist $f_p: D\to \Omega\cap U$ such that $f_p(0)=\eta_p$ tends to the origin and ${f_p}'(0)=R_p \dfrac{\overrightarrow{X}_p}{M(\eta_p, \overrightarrow{X}_p)},$ where $R_p\to \infty$ as $p\to \infty$. We may assume that $R_p\geq p^2$. Then, one has $M(f_p(0), {f'}_p(0))=M\biggl(\eta_p, R_p\dfrac{\overrightarrow{X}_p}{M(\eta_p, \overrightarrow{X}_p)}\biggl)=R_p\geq p^2$. Apply Lemma \ref{L6} to $M_p(t):=M(f_p(t)), {f_p}'(t))$ on $\bar D_{1/2}$ with $u=0$ and $\sigma=\frac{1}{p}$. This gives $\tilde a_p\in\bar D_{1/2} $ such that $|\tilde a_p|\leq \dfrac{2p}{M_p(0)}$ and $M_p(\tilde a_p)\geq M_p(0)\geq p^2$. Moreover, 
$$M_p(t)\leq 2 M_p(\tilde a_p) \ \text{on}\ D(\tilde a_p, \frac{p}{M_p(\tilde a_p)}).$$

We define a sequence $\{g_p\}\subset Hol(D_p, \Omega)$ by $g_p(t):=f_p\biggl(\tilde a_p+ \dfrac{At}{2 M_p(\tilde a_p)}\biggl)$. This sequence satisfies the estimates
$$M[g_p(t), {g_p}'(t)]\leq A \ \text{on}\  D_p. $$

Since $\tilde a_p\to 0$, the series $g_p(0)=f_p(\tilde a_p)$ tends to the origin. Choose a subsequence, if neccessary, we may assume that $K^p\epsilon(g_p(0))\leq \alpha_1,$  where $K, A$ and $\alpha_1$ are the constants in Lemma \ref{L4}. It follows from Lemma \ref{L4} that
\begin{equation}\label{Eq30}
g_p(D_N)\subset Q[g_p(0),K^N\epsilon(g_p(0))] \ \text{for}\  N\leq p. 
\end{equation}

We may now apply the method of dilation of the coordinates. Set $\eta_p:=g_p(0)$ and ${\eta'}_p:=\eta_p+(0,\cdots,0,\epsilon_p),$ where $\epsilon_p>0$ and $\rho({\eta'}_p)=0$. It is easy to see that $\epsilon_p\approx \epsilon(\eta_p) $ and $\eta_p \in Q({\eta'}_p,c\epsilon_p)$ for $c\geq 1$ is some constant. It follows from (\ref{Eq30}) and (\ref{Eq24}) that, for some constant $C\geq 1,$
\begin{equation}\label{Eq31}
g_p(D_N)\subset Q[{\eta'}_p, CK^N\epsilon_p] \ \text{for}\  N\leq p. 
 \end{equation}

Set $\varphi_p:=\Delta_{{\eta'}_p}^{\epsilon_p}\circ\Phi_{{\eta'}_p} \circ g_p$. The inclutions (\ref{Eq24}) imply that 

$$\varphi_p(D_N)\subset D_{\sqrt{C K^N}}\times\cdots\times D_{\sqrt{C K^N}}\times  D_{C K^N}.$$

By using the Montel's theorem and a diagonal process, there exists a subsequence $\{\varphi_{p_k}\}$ of $\{\varphi_{p}\}$ which converges on compact subsets of $\CC $ to an entire curve $\varphi:\CC\to M_P $. Since $M_P$ is Brody hyperbolic, $\varphi$ must be constant. 

On the other hand, we have 
\begin{equation*}
\begin{split}
\dfrac{A}{2}=M[g_p(0),{g_p}'(0)]=\sum_{k=1}^n \frac{|({\Phi'}_{\eta_p}(\eta_p){g_p}'(0))_k|}{\tau_k(\eta_p,\epsilon(\eta_p))}.
\end{split}
\end{equation*}

Since $\epsilon_p\approx \epsilon(\eta_p)\ ,\ \eta_p \in Q({\eta'}_p,c\epsilon_p)$ and ${\Phi'}_{\eta_p}(\eta_p)\circ \big({\Phi'}_{{\eta'}_p}(\eta_p)\big )^{-1}$ approaches to $Id$ as $p\to \infty $, we have 
\begin{equation*}
\begin{split}
\frac{A}{2}\lesssim \sum_{k=1}^n \frac{|({\Phi'}_{{\eta'}_p}(\eta_p){g_p}'(0))_k|}{\tau_k({\eta'}_p,\epsilon_p)}=\|{\varphi_p}'(0)\|_1.
\end{split}
\end{equation*}

Thus $\|{\varphi}'(0)\|_1=\lim_{p_k\to\infty}\|{\varphi_{p_k}}'(0)\|_1\gtrsim \frac{A}{2}$.
\end{proof}
\subsection{Normality of the families of holomorphic mappings}

First of all, we prove the following theorem. 
\begin{theorem}\label{T:5}
Let $\Omega$ be a domain in $\CC^n$. Suppose that $\partial \Omega$ is pseudoconvex, of finite type and is smooth of class $C^\infty$ near a boundary point $(0,\cdots,0)\in\partial \Omega.$ Suppose that the Levi form has rank at least $n-2$ at $(0,\cdots,0).$ Let $\omega $ be a domain in $\CC^k$ and $\varphi_p:\omega\to \Omega$ be a sequence of holomorphic mappings such that $\eta_p:=\varphi_p(a)$ converges to $(0,\cdots,0)$ for some point $a\in\omega$. Let $(T_p)_p$ be a sequence of automorphisms of $\CC^n$ which associates with the sequence $(\eta_p)_p$ by the  method of the dilation of coordinates (i.e,\ $T_p=\Delta_{{\eta'}_p}^{\epsilon_p}\circ \Phi_{{\eta'}_p}$). Then $(T_p\circ \varphi_p)_p$ is normal and its limits are holomorphic mappings from $\omega$ to the domain of the form 
$$M_P=\{(w_1,\cdots,w_n)\in \CC^n:Re  w_n+P(w_1,\bar w_1)+\sum_{\alpha=2}^{n-1}|w_\alpha|^2<0\},$$
where $P\in \P_{2m}$. 
\end{theorem}
\begin{proof}
Let $f: D\to \Omega$ be a holomorphic map with $f(0)$ near $(0,\cdots,0)$. By Theorem \ref{T2}, we have 
$$M[f(u),f'(u)]\lesssim K_\Omega(f(u),f'(u)) \leq K_D(u,\frac{\partial}{\partial u}).$$

Suppose $0<r_0<1$ such that $r_0\sup\limits_{|u|\leq r_0}K_D(u,\frac{\partial}{\partial u})\leq A$, where $A$ is the constant in Lemma \ref{L4}. Set $f_{r_0}(u):=f(r_0 u)$. Then
$$M[f_{r_0}(u),{f_{r_0}}'(u)]\leq A.$$

By Lemma \ref{L4}, we have $f(D_{r_0})=f_{r_0}(D)\subset Q[f(0), \epsilon(f(0))].$

This inclusion is also true if $D$ is replaced by the unit ball in $C^k$. Let $f:\omega \to \Omega$ be a holomorphic map such that 

$f(a)$ near $(0,\cdots,0)$  for some point $a\in\omega$. 
For any compact subset $K$ of $\omega$, by using a finite covering of balls of radius $r_0$ and by the property (\ref{Eq24}), we have  
$$f(K)\subset Q[f(a), C(K)\epsilon(f(a))],$$
where $C(K)$ is a constant which depends on $K$.

Since $\eta_p:=\varphi_p(a)$ converges to the origin, it implies that 
$$\varphi_p(K)\subset Q[{\eta'}_p, C(K)\epsilon(\eta_p)].$$ 

Thus $T_p\circ\varphi_p(K)\subset D_{\sqrt{C(K)}}\times\cdots\times D_{\sqrt{C(K)}} \times D_{C(K)}$. By the Montel's theorem and a diagonal process, the sequence $T_p\circ\varphi_p$ is normal and its limits are holomorphic mappings from $\omega$ to the domain of the form 
$$M_P=\{(w_1,\cdots,w_n)\in \CC^n:Re  w_n+P(w_1,\bar w_1)+\sum_{\alpha=2}^{n-1}|w_\alpha|^2<0\}.$$
\end{proof}

\section{Proof of Theorem \ref{T}}
\setcounter{theorem}{0}
\setcounter{nx}{0}
\setcounter{define}{0}
\setcounter{equation}{0}
  In this section, we use the Berteloot's method (see \cite{Ber2}) to complete the proof of Theorem \ref{T}. First of all, for a domain $\Omega$ in $\CC^n$ and $z\in \Omega$ we shall denote by $\P(\Omega,z)$ the set of polynomials $Q\in \P_{2m}$ such that $Q$ is subharmonic and there exists a biholomorphism $\psi : \Omega\to M_Q$ with $\psi(z)=(0',-1)$. By the similar argument as in the proof of Proposition 3.1 of \cite{Ber2} (also by using Theorem \ref{T:5} and Lemma \ref{L:4}), one also obtains that, if $\Omega$ satisfies the assumptions of our theorem, then $\P(\Omega,z)$ is never empty. Moreover, there are choices of $z$ such that every element of $\P(\Omega,z)$ is of degree $2m$. More precisely, we have the  following.
\begin{proposition} \label{P1} Let $\Omega$ be a domain in $\CC^n$ such that:
\begin{enumerate}
\item[(1)] $\exists \xi_0\in \partial \Omega$ such that $\partial \Omega$ is of class $C^\infty,$ pseudoconvex and of finite type in a neighbourhood of $\xi_0$.
\item[(2)] The Levi  form has rank at least $n-2$ at $\xi_0$. 
\item[(3)] $\exists z_0\in \Omega,\ \exists \varphi_p\in Aut(\Omega) $ such that $\lim\varphi_p(z_0)=\xi_0.$ 
\end{enumerate}

Then
\begin{enumerate}
\item[(a)] $\forall z\in \Omega: \P(\Omega,z)\ne \emptyset$. 
\item[(b)] $\exists \tilde z_0 \in \Omega$ such that if $Q\in \P(\Omega,\tilde z_0),$ then $\deg Q=2m,$ where $2m$ is the type of $\partial \Omega$ at $\xi_0.$ 
\item[(c)] $\exists Q\in \P(\Omega,\tilde z_0)$ such that $Q=H+R,$ where $H\in \H_{2m}$ and $\deg R<2m$.
\end{enumerate}  
\end{proposition}

The control of sequence of dilations associated to the "orbit" $(\varphi_p(\tilde z_0))$ is closely related to the asymptotic behaviour of $(\varphi_p(\tilde z_0))$ in $\Omega$. Unfortunately, the direct investigation of this behaviour seems impossible. Our aim is therefore to study the image of $(\varphi_p(\tilde z_0))$ in some rigid polynomial realization $M_Q$ of $\Omega$. The proof of our theorem follows from the following proposition which summarizes the different possibilities.     
\begin{proposition} \label{P2}Let $\Omega$ be a domain in $\CC^n$ satisfying the following assumptions:
\begin{enumerate}
 \item[(1)]$\partial \Omega$ is smoothly pseudoconvex in a neighbourhood of $\xi_0\in \partial \Omega$ and of finite type $2m$ at $\xi_0$.
 \item[(2)]$\exists z_0\in \Omega,\ \exists \varphi_p\in Aut(\Omega) $ s.t. $\lim\varphi_p(z_0)=\xi_0.$ 
Let $\tilde z_0\in \Omega$ and $Q\in \P(\Omega,\tilde z_0) $ be given by Proposition \ref{P1} and let $\psi$ denote a biholomorphism between $\Omega$ and $M_Q$ which maps $\tilde z_0$ onto $(0',-1)$; denote $\psi\circ\varphi_p(\tilde z_0)  $ as $a_p=(a_{1p},\cdots,a_{np})$ and $|Re\psi_n\circ\varphi_p(\tilde z_0)+Q[\psi_1\circ\varphi_p(\tilde z_0)]+|\psi_2\circ\varphi_p(\tilde z_0)|^2+\cdots+|\psi_{n-1}\circ\varphi_p(\tilde z_0)|^2|$ as $\epsilon_p$. Let $H$ be the homogeneous part of highest degree in $Q.$
\end{enumerate}
Then three possibilities may occur
 \begin{enumerate}
 \item[(i)] $\lim\epsilon_p=0$ and $\lim \inf |a_{1p}|<+\infty$.\\
 Then $Q(z)=H(z-a)+2 Re\sum_{j=0}^{2m}\frac{Q_j(a)}{j!}(z-a)^j (a\in \CC)$ and $\Omega\simeq M_H$.
  \item[(ii)] $\lim \epsilon_p=0$ and $\liminf |a_{1p}|=+\infty$.\\
Then $Q(z)=H=\lambda [(2 Re( e^{i\nu}z))^{2m}-2Re(e^{i\nu}z)^{2m}](\lambda>0,\ \nu\in [0,2\pi))$ and $\Omega\simeq M_H$
\item[(iii)]$\limsup \epsilon_p>0$. Then $H=\lambda |z|^{2m}(\lambda>0) $ and $\Omega\simeq M_H$
\end{enumerate}
\end{proposition}
\begin{proof}
We may assume that $\deg Q>2$. Otherwise $Q=|z|^{2}$ and the theorem already follows from Proposition \ref{P1}. Let us first consider the case where $\lim\epsilon_p=0$. Define a sequence of polynomials $Q_p$ by 
\begin{equation}\label{e1}
Q_p=\dfrac{1}{\epsilon_p}\sum_{j,q>0}\frac{Q_{j,\bar q}(a_{1p})}{(j+q)!}\tau_p^{j+q}z_1^j\bar z_1^q
\end{equation}
where $\tau_p>0$ is chosen in order to achieve $\|Q_p\|=1$. Taking a sequence we may assume that $\lim Q_p=Q_\infty$ where $Q_\infty\in\P_{2m}$ and $\|Q_\infty\|=1$. 

Let us consider the sequence of automorphisms of $\CC^n$
\begin{equation*}
\begin{split}
\phi_p: \CC^n&\to \CC^n\\ 
z &\mapsto z', 
\end{split}
\end{equation*}
where $z'$ is given by  

\begin{equation}\label{e2}
\begin{cases}
{z'}_n=\dfrac{1}{\epsilon_p}\Big[z_n-a_{np}-\epsilon_p+2\sum\limits_{j=1}^{2m}\frac{Q_j(a_{1p})}{j!}(z_1-a_{1p})^j +2\sum\limits_{j=2}^{n-1}\bar{a}_{jp}(z_j-a_{jp})       \Big]\\
{z'}_1 =\dfrac{1}{\tau_p}[z_1-a_{1p}]\\
{z'}_2 =\dfrac{1}{\sqrt{\epsilon_p}}[z_2-a_{2p}]\\
...\\
{z'}_{n-1} =\dfrac{1}{\sqrt{\epsilon_p}}[z_{n-1}-a_{n-1p}]
\end{cases}
\end{equation}

It is easy to check that $\phi_p$ maps biholomorphically $M_Q$ onto $M_{Q_p}$ and $a_p$ to $(0',-1)$. 

i) and ii) are now obtained with a slightly modification of the proof of Proposition 4.1 in \cite{Ber2}. We are going to prove iii).

We now consider the case where $\limsup\epsilon_p>0$. After taking some subsequence we may assume that $\epsilon_p\geq c>0$ for all $p$. We shall study the real action $(g_t)$ defined on $M$ by
 \begin{equation}\label{e5}
\begin{cases}
g: \RR\times \Omega\to \Omega\\
(t,z)\mapsto g_t(z)\\
g_t(z)=\psi^{-1}[\psi(z)+(0',it)].
\end{cases}
\end{equation}
Modifying the proof of Lemma 4.3 of \cite{Ber2}, we also conclude that this action is a parabolicity, that is 
\begin{equation}\label{e6}
\forall z\in \Omega: \lim_{t\to \pm\infty}g_ t(z)=\xi_0.
\end{equation}
According to \cite{B-P2}, the action $(g_t)_t$ itself is of class $C^\infty$. Thus, we may now consider the holomorphic tangent vector field $\vec{X}$ defined on some neighbourhood of $\xi_0$ in $\partial \Omega$ by 
$$\vec{X}=\big(\frac{d}{dt}\big)_{t=0} g_t(z).$$

The analysis of this vector field is given in the papers of E. Bedford and S. Pinchuk \cite{B-P1}, \cite{B-P2}. It yields the conclution that $H=|z|^{2m}$. It is then possible to study the scaling process more precisely for showing that $\Omega$ is biholomorphic to $M_{|z|^{2m}}$. This ends the proof of Proposition \ref{P2}
\end{proof}

\end{document}